\documentclass[a4paper]{amsart}
\usepackage{graphicx,amsmath,amsfonts,latexsym,amssymb,amsthm,mathrsfs, color}
\usepackage[latin1]{inputenc}
\newtheorem{theorem}{Theorem}[section]

\newtheorem{definition}{Definition}[section]

\newtheorem{example}{Example}[section]

\begin{document}
\title%
{On the Inverse of the sum of two sectorial operators}
\author{Nikolaos Roidos}
\address{Institut für Analysis, Leibniz Universität Hannover, Welfengarten 1, 
30167 Hannover, Germany}
\email{roidos@math.uni-hannover.de}

\begin{abstract}
We study an abstract linear operator equation on a Banach space by using the inverse of the sum of two sectorial operators. We prove that the boundedness of a special type of operator valued $H^\infty$-calculus is sufficient for maximal regularity of the solution. We apply the result to the abstract parabolic problem, to give a maximal $L^{p}$-regularity condition. We also study the abstract hyperbolic problem and give a sufficient condition for the existence of solution. 
\end{abstract}

\date{\today}

\maketitle

\section{introduction}

We consider an abstract linear operator equation on a Banach space $E$ of the form
\begin{gather}\label{eb}
(A+B)x=y,
\end{gather}
where $A$, $B$ are closed linear (resolvent) commuting operators. The importance of the above equation is that instead of considering the equation on $E$ we can consider it on the $E$-valued $L^{p}$ space $L^{p}(0,T;E)$, for some $T>0$ and $p\in(1,\infty)$, and take $B$ to be the first or the second derivative with respect to $t\in(0,T)$, with appropriate boundary conditions, to obtain the abstract parabolic and hyperbolic problem respectively (i.e. the first and second order abstract Cauchy problem). In such problems we are interested in sufficient conditions for the existence of a solution and also the regularity of solution. 
\par
Da Prato and Grisvard in \cite{PG} showed that the sectoriality property for the operators $A$ and $B$ (i.e. the good asymptotic behavior of the resolvents in some sectors) is sufficient for the existence of a unique solution of (\ref{eb}) for any $y\in E$, which also depends continuously on $y$. From the analysis there, it follows that the closedness of the sum of the two closed operators $A$ and $B$ is connected to the \textit{maximal regularity} of the solution, i.e. to the fact that the solution belongs to the intersection of the domains $\mathcal{D}(A)\cap\mathcal{D}(B)$. The last property plays an important role in the approach of the nonlinear problems (see  Clément and Li, \cite{CL}). Dore and Venni in \cite{DV} studied the problem (\ref{eb}) in the case of a UMD space, for sectorial operators having bounded imaginary powers, and gave a sufficient condition for the closedness of the sum and hence for the maximal regularity of the solution. An application to the first derivative was also given. Kalton and Weis in \cite{KW}, by using the connection between the boundedness of the joint functional calculus and the closedness of the sum for two sectorial operators (see Proposition 2.7 in \cite{LA}), gave another answer to the maximal regularity question, by requiring bounded $H^{\infty}$-calculus to one of the operators and Rademacher boundedness to the other. Finally, Neidhardt and Zagrebnov in \cite{NZ} treated the abstract parabolic problem in a more general case (non autonomous evolution equation) with a different approach, i.e. by extending a certain evolution operator to an anti-genarator of an evolution semigroup, and provided existence together with regularity results.   
\par
In section 2 of this paper we study the problem (\ref{eb}) in a classical sense, by using a formula for the inverse of the sum of two operators. The formula we use is the same as in \cite{PG}. Nevertheless, we apply a different approach by means of complex powers of the operators. Namely, we regard the inverse of the sum of the operators restricted to the images of complex powers with negative real part. In this way, we prove the same results as in \cite{PG} concerning the interpolation spaces. Moreover, in a similar approach, by considering the inverse of the sum on the image of the bounded holomorphic semigroup generated by at least one of the two operators, we see (Theorem \ref{t1}) that we can impose a similar condition for the closedness of the sum to that one of the bounded operator valued $H^{\infty}$-calculus. Hence, we find that a sufficient condition for the sum to be closed is that one of the two operators, which has to be a generator of a bounded holomorphic semigroup, has to admit a special type of bounded $H^{\infty}$-calculus for operator-valued holomorphic functions of exponential decay. In the third section, we apply the abstract result to the first derivative. In section four we study the abstract hyperbolic problem. Since the second derivative is not a sectorial operator, we treat the problem by defining the inverse of the sum in a special way. We show that a sufficient condition for the existence of a solution is that in addition to the classical sectoriality property, the resolvent of the operator has to satisfy some decay condition on the right hand side of a parabola (see Definition \ref{dp}).

\section{The inverse of A+B}

\begin{definition}(Sectorial operators)
Let $E$ be a Banach space, $K\geq1$ and $\theta\in[0,\pi)$. Let $\mathcal{P}_{K}(\theta)$ be the class of closed linear operators in $E$ such that if $A\in\mathcal{P}_{K}(\theta)$, then 
\[
S_{\theta}=\{z\in\mathbb{C}\,|\, |\arg z|\leq\theta\}\cup\{0\}\subset\rho{(-A)} \,\,\,\,\,\, \mbox{and} \,\,\,\,\,\, (1+|z|)\|(A+z)^{-1}\|\leq K, \,\,\,\,\,\, \forall z\in S_{\theta}.
\]
Also, let $\mathcal{P}(\theta)=\cup_{K}\mathcal{P}_{K}(\theta)$.
\end{definition}
If $A\in \mathcal{P}(\theta)$, then by a sectoriality extension argument (see the Appendix) we can always assume that $\theta>0$.
\begin{definition}
For any $\rho\geq0$ and $\theta\in(0,\pi)$, let $\Gamma_{\rho,\theta}$ be the positively oriented path 
\[
\{\rho e^{i\phi}\in\mathbb{C}\,|\,\theta\leq\phi\leq2\pi-\theta\}\cup\{re^{\pm i\theta}\in\mathbb{C}\,|\,r\geq\rho\}.
\]
If $\rho=0$, we denote $\Gamma_{\rho,\theta}$ by $\Gamma_{\theta}$.
\end{definition}
We define next a special type of bounded $H^{\infty}$-calculus, for holomprphic operator valued families which decay exponentially in the complement of the sector of a sectorial operator.
\begin{definition}
Let $E$ be a Banach space and $A\in \mathcal{P}(\theta)$, $\theta>\pi/2$. Let $H^{e,\infty}_{\mathcal{L}(E)}(\theta)$ be the space of all bounded holomorphic functions $f:\mathbb{C}\setminus S_{\theta}\rightarrow \mathcal{L}(E)$ such that $f(\lambda)$ and $(A+z)^{-1}$ commute for all $\lambda\in \mathbb{C}\setminus S_{\theta}$ and $z\in S_{\theta}$, and 
\[
\|f(\lambda)\|_{\mathcal{L}(E)}\leq c\frac{|\lambda|}{1+|\lambda|}e^{-\delta|\lambda|}, \,\,\, \mbox{for any} \,\,\, \lambda\in \mathbb{C}\setminus S_{\theta}, 
\]
and some $c$, $\delta>0$ depending on $f$. Any $f\in H^{e,\infty}_{\mathcal{L}(E)}(\theta)$ can be extended to non-tangential values in $\partial S_{\theta}$, and defines an element in $\mathcal{L}(E)$ by 
\[
f(-A)=\frac{1}{2\pi i}\int_{\Gamma_{\theta}}f(\lambda)(A+\lambda)^{-1} d\lambda.
\]
We say that $A$ admits a bounded $H^{e,\infty}_{\mathcal{L}(E)}(\theta)$-calculus if $\|f(-A)\|_{\mathcal{L}(E)}\leq C_{A}\|f\|_{\infty}$ for any $f\in H^{e,\infty}_{\mathcal{L}(E)}(\theta)$, where $C_{A}$ is independent of $f$ and $\|f\|_{\infty}$ is the supremum norm of $\|f(\lambda)\|_{\mathcal{L}(E)}$.
\end{definition}
Since we regard only commuting operators, we recall the following definition.
\begin{definition}
Two closed linear operators $A$, $B$ in a Banach space $E$ are \textsl{resolvent commuting} if there exist some $\lambda\in\rho(-A)$ and $\mu\in\rho(-B)$ such that 
\[
[(A+\lambda)^{-1},(B+\mu)^{-1}]=0.
\]
\end{definition}
At the following we will use (without mention it) Lemma III.4.9.1 in \cite{Am}. The first part of next theorem is contained in the results of \cite{PG}. 

\begin{theorem}\label{t1}
Let $E$ be a Banach space, $A\in\mathcal{P}(\theta_{A})$ and $B\in\mathcal{P}(\theta_{B})$ be resolvent commuting with $\theta_{A}>\theta_{B}$ and $\theta_{A}+\theta_{B}>\pi$. Then, $A+B$ with $\mathcal{D}(A+B)=\mathcal{D}(A)\cap\mathcal{D}(B)$ is closable and the following equation
\[
\overline{{(A+B)}}x=y,
\] 
for any $y\in E$, has a unique solution $x\in\bigcap_{\theta<1}(E,\mathcal{D}(A))_{\theta,q}\cap(E,\mathcal{D}(B))_{\theta,q}$, for any $q\in[1,\infty)$, given by 
\[
x=\frac{1}{2\pi i}\int_{\Gamma_{\theta_{B}}}(A-z)^{-1}(B+z)^{-1}ydz.
\]
If  
\[
y\in\bigcup_{\theta>0}(E,\mathcal{D}(A))_{\theta,p}\cup(E,\mathcal{D}(B))_{\theta,p},
\]
for some $p\in[1,\infty)$, then $x\in\mathrm{D}(A)\cap\mathrm{D}(B)$.\\
Moreover, if $A$ admits a bounded $H^{e,\infty}_{\mathcal{L}(E)}(\theta_{A})$-calculus, then $x\in\mathrm{D}(A)\cap\mathrm{D}(B)$ for any $y\in E$ and
\[
\overline{{(A+B)}}=A+B.
\]
\end{theorem}

\begin{proof}
Let the bounded in $E$ operator
\begin{gather}\label{K}
\mathcal{K}=\frac{1}{2\pi i}\int_{\Gamma_{\theta_{B}}}(A-z)^{-1}(B+z)^{-1}dz.
\end{gather}
By a sectoriality extension argument and Cauchy's theorem, in the above integral formula we can replace the path $\Gamma_{\theta_{B}}$ by $c+\Gamma_{\theta_{B}-\varepsilon}$ for some $c\in\mathbb{R}$ and $\varepsilon>0$ sufficiently closed to zero.
\par
Let $y\in(E,\mathcal{D}(A))_{\theta',p}$ for some $0<\theta'<1$ and $p\in[1,\infty)$. Then, by I.2.5.2 and I.2.9.6 in \cite{Am}, $y\in\mathcal{D}(A^{\theta})$ for any $0<\theta<\theta'$. If we take $c<0$, then by the standard way of defining fractional powers for sectorial operators (see Theorem III.4.6.5 in \cite{Am}), for sufficiently small $\rho>0$ and the path $\Gamma_{\rho,\theta_{A}}$, we have that
\begin{eqnarray*}
\lefteqn{\mathcal{K}y=\frac{1}{2\pi i}\int_{c+\Gamma_{\theta_{B}}}(A-z)^{-1}(B+z)^{-1}A^{-\theta}A^{\theta}ydz}\\
&=&\frac{1}{2\pi i}\int_{c+\Gamma_{\theta_{B}}}(A-z)^{-1}(B+z)^{-1}\big(\frac{1}{2\pi i}\int_{\Gamma_{\rho,\theta_{A}}}(-\lambda)^{-\theta}(A+\lambda)^{-1}d\lambda\big)A^{\theta}ydz\\
&=&(\frac{1}{2\pi i})^{2}\int_{c+\Gamma_{\theta_{B}}}\int_{\Gamma_{\rho,\theta_{A}}}(B+z)^{-1}(-\lambda)^{-\theta}(A-z)^{-1}(A+\lambda)^{-1}A^{\theta}yd\lambda dz\\
&=&(\frac{1}{2\pi i})^{2}\int_{c+\Gamma_{\theta_{B}}}\int_{\Gamma_{\rho,\theta_{A}}}(B+z)^{-1}(-\lambda)^{-\theta}(z+\lambda)^{-1}\big((A-z)^{-1}-(A+\lambda)^{-1}\big)A^{\theta}yd\lambda dz\\
&=&(\frac{1}{2\pi i})^{2}\int_{c+\Gamma_{\theta_{B}}}\int_{\Gamma_{\rho,\theta_{A}}}(A-z)^{-1}(B+z)^{-1}(-\lambda)^{-\theta}(z+\lambda)^{-1}A^{\theta}yd\lambda dz\\
&&-(\frac{1}{2\pi i})^{2}\int_{\Gamma_{\rho,\theta_{A}}}\int_{c+\Gamma_{\theta_{B}}}(A+\lambda)^{-1}(B+z)^{-1}(-\lambda)^{-\theta}(z+\lambda)^{-1}A^{\theta}ydzd\lambda, 
\end{eqnarray*}
where at the last step we have used Fubini's theorem. By Cauchy's theorem, the first term in the right hand side of the above equation is zero, and hence
\begin{gather}\label{e111}
\mathcal{K}y=\frac{1}{2\pi i}\int_{-\Gamma_{\rho,\theta_{A}}}(A-\lambda)^{-1}(B+\lambda)^{-1}\lambda^{-\theta}A^{\theta}yd\lambda.
\end{gather}
Since for any sectorial operator $T$, $T(T+z)^{-1}=I-z(T+z)^{-1}$ is uniformly bounded in $z$ inside the area of the sectoriality, the integrals 
\begin{gather*}
\int_{-\Gamma_{\rho,\theta_{A}}}A(A-\lambda)^{-1}(B+\lambda)^{-1}\lambda^{-\theta}A^{\theta}yd\lambda \,\,\, \mbox{and} \,\,\, \int_{-\Gamma_{\rho,\theta_{A}}}(A-\lambda)^{-1}B(B+\lambda)^{-1}\lambda^{-\theta}A^{\theta}yd\lambda
\end{gather*}
converge absolutely. Hence, by (\ref{e111}), $\mathcal{K}y\in\mathcal{D}(A)\cap\mathcal{D}(B)$ and 
\begin{gather*}
A\mathcal{K}y=\frac{1}{2\pi i}\int_{-\Gamma_{\rho,\theta_{A}}}A(A-\lambda)^{-1}(B+\lambda)^{-1}\lambda^{-\theta}A^{\theta}yd\lambda,\\ B\mathcal{K}y=\frac{1}{2\pi i}\int_{-\Gamma_{\rho,\theta_{A}}}(A-\lambda)^{-1}B(B+\lambda)^{-1}\lambda^{-\theta}A^{\theta}yd\lambda.
\end{gather*}
Thus, 
\begin{eqnarray*}
\lefteqn{(A+B)\mathcal{K}y=\frac{1}{2\pi i}\int_{-\Gamma_{\rho,\theta_{A}}}(A+B)(A-\lambda)^{-1}(B+\lambda)^{-1}\lambda^{-\theta}A^{\theta}yd\lambda}\\
&=&\frac{1}{2\pi i}\int_{-\Gamma_{\rho,\theta_{A}}}(A-\lambda+B+\lambda)(A-\lambda)^{-1}(B+\lambda)^{-1}\lambda^{-\theta}A^{\theta}yd\lambda\\
&=&\frac{1}{2\pi i}\int_{-\Gamma_{\rho,\theta_{A}}}(B+\lambda)^{-1}\lambda^{-\theta}A^{\theta}yd\lambda+\frac{1}{2\pi i}\int_{-\Gamma_{\rho,\theta_{A}}}(A-\lambda)^{-1}\lambda^{-\theta}A^{\theta}yd\lambda.
\end{eqnarray*}
The first term in the right hand site on the above equation is zero by Cauchy's theorem. Hence, by the definition of the complex powers for $A$, we find that
\begin{gather*}
(A+B)\mathcal{K}y=A^{-\theta}A^{\theta}y=yI.
\end{gather*}

\par
Let now that $y\in(E,\mathcal{D}(B))_{\theta',p}$. Then, similarly to the previous case, $y\in\mathcal{D}(B^{\theta})$. If we take $c>0$, we have that 
\begin{eqnarray*}
\lefteqn{\mathcal{K}y=\frac{1}{2\pi i}\int_{c+\Gamma_{\theta_{B}-\varepsilon}}(A-z)^{-1}(B+z)^{-1}B^{-\theta}B^{\theta}ydz}\\
&=&\frac{1}{2\pi i}\int_{c+\Gamma_{\theta_{B}-\varepsilon}}(A-z)^{-1}(B+z)^{-1}\big(\frac{1}{2\pi i}\int_{\Gamma_{\rho,\theta_{B}}}(-\lambda)^{-\theta}(B+\lambda)^{-1}d\lambda\big)B^{\theta}ydz\\
&=&(\frac{1}{2\pi i})^{2}\int_{c+\Gamma_{\theta_{B}-\varepsilon}}\int_{\Gamma_{\rho,\theta_{B}}}(A-z)^{-1}(-\lambda)^{-\theta}(B+z)^{-1}(B+\lambda)^{-1}B^{\theta}yd\lambda dz\\
&=&(\frac{1}{2\pi i})^{2}\int_{c+\Gamma_{\theta_{B}-\varepsilon}}\int_{\Gamma_{\rho,\theta_{B}}}(A-z)^{-1}(-\lambda)^{-\theta}(\lambda-z)^{-1}\big((B+z)^{-1}-(B+\lambda)^{-1}\big)B^{\theta}yd\lambda dz\\
&=&(\frac{1}{2\pi i})^{2}\int_{c+\Gamma_{\theta_{B}-\varepsilon}}\int_{\Gamma_{\rho,\theta_{B}}}(A-z)^{-1}(B+z)^{-1}(-\lambda)^{-\theta}(\lambda-z)^{-1}B^{\theta}yd\lambda dz\\
&&-(\frac{1}{2\pi i})^{2}\int_{\Gamma_{\rho,\theta_{B}}}\int_{c+\Gamma_{\theta_{B}-\varepsilon}}(A-z)^{-1}(B+\lambda)^{-1}(-\lambda)^{-\theta}(\lambda-z)^{-1}B^{\theta}ydzd\lambda, 
\end{eqnarray*}
where at the last step we have used again Fubini's theorem. By Cauchy's theorem we find that
\begin{gather}\label{e2}
\mathcal{K}y=
\frac{1}{2\pi i}\int_{\Gamma_{\rho,\theta_{B}}}(A-\lambda)^{-1}(B+\lambda)^{-1}(-\lambda)^{-\theta}B^{\theta}yd\lambda. 
\end{gather}
Since both integrals 
\begin{gather*}
\int_{\Gamma_{\rho,\theta_{B}}}A(A-\lambda)^{-1}(B+\lambda)^{-1}(-\lambda)^{-\theta}B^{\theta}yd\lambda, \,\,\, \int_{\Gamma_{\rho,\theta_{B}}}(A-\lambda)^{-1}B(B+\lambda)^{-1}(-\lambda)^{-\theta}B^{\theta}yd\lambda
\end{gather*}
converge absolutely, we have that $\mathcal{K}y\in\mathcal{D}(A)\cap\mathcal{D}(B)$ and 
\begin{gather*}
A\mathcal{K}y=
\frac{1}{2\pi i}\int_{\Gamma_{\rho,\theta_{B}}}A(A-\lambda)^{-1}(B+\lambda)^{-1}(-\lambda)^{-\theta}B^{\theta}yd\lambda,\\
B\mathcal{K}y=
\frac{1}{2\pi i}\int_{\Gamma_{\rho,\theta_{B}}}(A-\lambda)^{-1}B(B+\lambda)^{-1}(-\lambda)^{-\theta}B^{\theta}yd\lambda.  
\end{gather*}
Hence,
\begin{eqnarray*}
\lefteqn{(A+B)\mathcal{K}y=
\frac{1}{2\pi i}\int_{\Gamma_{\rho,\theta_{B}}}(A+B)(A-\lambda)^{-1}(B+\lambda)^{-1}(-\lambda)^{-\theta}B^{\theta}yd\lambda}\\ 
&=&\frac{1}{2\pi i}\int_{\Gamma_{\rho,\theta_{B}}}(A-\lambda+B+\lambda)(A-\lambda)^{-1}(B+\lambda)^{-1}(-\lambda)^{-\theta}B^{\theta}yd\lambda\\
&=&\frac{1}{2\pi i}\int_{\Gamma_{\rho,\theta_{B}}}(B+\lambda)^{-1}(-\lambda)^{-\theta}B^{\theta}yd\lambda+\frac{1}{2\pi i}\int_{\Gamma_{\rho,\theta_{B}}}(A-\lambda)^{-1}(-\lambda)^{-\theta}B^{\theta}yd\lambda.
\end{eqnarray*}
The last term in the above equation is zero by Cauchy's theorem, and by the definition of the complex powers for $B$, we find that
\begin{gather*}
(A+B)\mathcal{K}y=B^{-\theta}B^{\theta}y=yI.
\end{gather*}

\par 
Now take $\phi\in(0,1)$ and any $y\in E$. Then, by (\ref{e111}), there is
\begin{gather*}
A^{\phi-1}\mathcal{K}y=\mathcal{K}A^{\phi-1}y
=\frac{1}{2\pi i}\int_{-\Gamma_{\rho,\theta_{A}}}(A-z)^{-1}(B+z)^{-1}z^{\phi-1}ydz.
\end{gather*}
Since the integral
\begin{gather*}
\int_{-\Gamma_{\rho,\theta_{A}}}A(A-z)^{-1}(B+z)^{-1}z^{\phi-1}ydz
\end{gather*}
converges absolutely, we have that $A^{\phi-1}\mathcal{K}y\in\mathcal{D}(A)$, which implies that $\mathcal{K}y\in\mathcal{D}(A^{\phi})$ (by the properties of the complex powers of an operator, cf. Theorem III.4.6.5 in \cite{Am} ). Thus,
by I.2.9.6 and I.2.5.2 in \cite{Am}, we have that $\mathcal{K}y\in(E,\mathcal{D}(A))_{\phi',q}$ for any $0<\phi'<\phi$ and any $q\in[1,\infty)$. 

\par
Similarly, by (\ref{e2}) we have that
\begin{gather*}
B^{\phi-1}\mathcal{K}y=\frac{1}{2\pi i}\int_{\Gamma_{\rho,\theta_{B}}}(A-z)^{-1}(B+z)^{-1}(-z)^{\phi-1}ydz.
\end{gather*}
Since the integral 
\begin{gather*}
\int_{\Gamma_{\rho,\theta_{B}}}(A-z)^{-1}B(B+z)^{-1}(-z)^{\phi-1}ydz
\end{gather*}
converges absolutely, we obtain that $B^{\phi-1}\mathcal{K}y\in\mathcal{D}(B)$, or that $\mathcal{K}y\in\mathcal{D}(B^{\phi})$. Hence, we find as before that $\mathcal{K}y\in(E,\mathcal{D}(A))_{\phi',q}$.

\par
Let $\{x_{n}\}_{n\in\mathbb{N}}$ be a sequence in $\mathcal{D}(A+B)$ such that $x_{n}\rightarrow 0$ and $(A+B)x_{n}\rightarrow y$ as $n\rightarrow\infty$. There is
\begin{gather*}
x_{n}=(A+B)\mathcal{K}x_{n}=\mathcal{K}(A+B)x_{n}\rightarrow \mathcal{K}y,
\end{gather*}
which implies that $\mathcal{K}y=0$. By the relation
\begin{gather}\label{e474}
A^{-1}B^{-1}w=(A+B)\mathcal{K}A^{-1}B^{-1}w=(A^{-1}+B^{-1})\mathcal{K}w, \,\,\, \forall w\in E,
\end{gather}
we find that $y=0$, or that $A+B$ is closable. That $\overline{{(A+B)}}\mathcal{K}y=y$, for any $y\in E$, follows by the density of $\bigcup_{\theta>0}(E,\mathcal{D}(A))_{\theta,p}\cup(E,\mathcal{D}(B))_{\theta,p}$ in $E$, for any $p\in[1,\infty)$.
\par
For the closedness of the sum of the two operators, let that $\{\tilde{x}_{n}\}_{n\in\mathbb{N}}$ be a sequence in $\mathcal{D}(A+B)$ such that $\tilde{x}_{n}\rightarrow \tilde{x}$ and $(A+B)\tilde{x}_{n}\rightarrow \tilde{y}$ as $n\rightarrow\infty$. By applying $\mathcal{K}$ to the last limit, we find that $\tilde{x}=\mathcal{K}\tilde{y}$. If we show that $\mathcal{K}$ maps to $\mathcal{D}(A+B)$, then since $(A+B)\mathcal{K}=I$ in the dense set $\bigcup_{\theta>0}(E,\mathcal{D}(A))_{\theta,p}\cup(E,\mathcal{D}(B))_{\theta,p}$, we will have that the sum $A+B$ is closed (it follows alternatively by (\ref{e474})). By (\ref{e474}), it is enough to show that $\mathcal{K}$ maps to one of the domains $\mathcal{D}(A)$ or $\mathcal{D}(B)$.
\par
Since $\theta_{A}>\pi/2$, $A$ generates a bounded holomorphic semigroup on $E$, which is defined by 
\begin{gather*}
e^{-wA}=\frac{1}{2\pi i}\int_{\Gamma_{\theta_{A}}}e^{w\lambda}(A+\lambda)^{-1}d\lambda, \,\,\, \mbox{with} \,\,\, |\arg w|\leq\theta_{A}-\frac{\pi}{2}.
\end{gather*}
For any $y\in E$, and $c<0$ sufficiently close to zero, by Fubini's theorem, we have that
\begin{eqnarray*}
\lefteqn{\mathcal{K}e^{-wA}y=(\frac{1}{2\pi i})^{2}\int_{c+\Gamma_{\theta_{B}}}(A-z)^{-1}(B+z)^{-1}(\int_{\Gamma_{\theta_{A}}}e^{w\lambda}(A+\lambda)^{-1}yd\lambda)dz}\\
&=&(\frac{1}{2\pi i})^{2}\int_{c+\Gamma_{\theta_{B}}}\int_{\Gamma_{\theta_{A}}}(B+z)^{-1}e^{w\lambda}(z+\lambda)^{-1}\big((A-z)^{-1}-(A+\lambda)^{-1}\big)yd\lambda dz\\
&=&(\frac{1}{2\pi i})^{2}\int_{c+\Gamma_{\theta_{B}}}\int_{\Gamma_{\theta_{A}}}(A-z)^{-1}(B+z)^{-1}e^{w\lambda}(z+\lambda)^{-1}yd\lambda dz\\
&&-(\frac{1}{2\pi i})^{2}\int_{\Gamma_{\theta_{A}}}\int_{c+\Gamma_{\theta_{B}}}(A+\lambda)^{-1}(B+z)^{-1}e^{w\lambda}(z+\lambda)^{-1}ydzd\lambda.
\end{eqnarray*}
By Cauchy's theorem, the first term in the right hand side of the above equation is zero, and hence
\begin{gather*}
\mathcal{K}e^{-wA}y=\frac{1}{2\pi i}\int_{\Gamma_{\theta_{A}}}(A+\lambda)^{-1}(B-\lambda)^{-1}e^{w\lambda}yd\lambda.
\end{gather*}
Since the integral
\begin{gather*}
\int_{\Gamma_{\theta_{A}}}A(A+\lambda)^{-1}(B-\lambda)^{-1}e^{w\lambda}yd\lambda
\end{gather*}
converges absolutely, $\mathcal{K}e^{-wA}y\in\mathcal{D}(A)$ and 
\begin{eqnarray}\nonumber
\lefteqn{A\mathcal{K}e^{-wA}y=\frac{1}{2\pi i}\int_{\Gamma_{\theta_{A}}}A(A+\lambda)^{-1}(B-\lambda)^{-1}e^{w\lambda}yd\lambda}\\\nonumber
&=&\frac{1}{2\pi i}\int_{\Gamma_{\theta_{A}}}(A+\lambda-\lambda)(A+\lambda)^{-1}(B-\lambda)^{-1}e^{w\lambda}yd\lambda\\\nonumber
&=&\frac{1}{2\pi i}\int_{\Gamma_{\theta_{A}}}(B-\lambda)^{-1}e^{w\lambda}yd\lambda-\frac{1}{2\pi i}\int_{\Gamma_{\theta_{A}}}(A+\lambda)^{-1}(B-\lambda)^{-1}\lambda e^{w\lambda}yd\lambda\\\label{eee}
&=&-\frac{1}{2\pi i}\int_{\Gamma_{\theta_{A}}}(A+\lambda)^{-1}(B-\lambda)^{-1}\lambda e^{w\lambda}yd\lambda,
\end{eqnarray}
where we have used again Cauchy's theorem.
\par
Assume that the operator $A$ admits a bounded $H^{e,\infty}_{\mathcal{L}(E)}(\theta_{A})$-calculus. If we restrict $y\in\mathcal{D}(A)$ in (\ref{eee}) and take the limit $w\rightarrow0$ with $w\in\mathbb{R}$, since $A\mathcal{K}e^{-wA}y=e^{-wA}A\mathcal{K}y$, we find that $\|A\mathcal{K}y\|\leq C \|y\|$ for some $C>0$ depending only on $A$ and $B$. By a Cauchy's sequence argument and the closedness of $A$, we see that $\mathcal{K}$ maps to $\mathcal{D}(A)$.
\end{proof}

\section{The abstract parabolic problem}

Let the operator $B=\partial_{t}$ in $L^{p}(0,T;E)$ with $\mathcal{D}(B)=\{f(t)\in W^{1,p}(0,T;E)\,|\,f(0)=0\}$, for some $p\in(1,\infty)$ and $T>0$. We have that $\sigma(B)=\emptyset$ and 
\begin{gather}\label{e778}
(B+\lambda)^{-1}g=\int_{0}^{t}e^{\lambda(x-t)}g(x)dx, \,\,\, \forall \lambda\in\mathbb{C},
\end{gather} 
where by the Young's inequality for convolution, we infer that
\begin{gather*}
\|(B+\lambda)^{-1}\|\leq\frac{1-e^{-\mathrm{Re}(\lambda) T}}{\mathrm{Re}(\lambda)}, \,\,\, \forall \lambda\in\mathbb{C}.
\end{gather*} 
If we extend $A:\mathcal{D}(A)\rightarrow E$ to $A: L^{p}(0,T;\mathcal{D}(A))\rightarrow L^{p}(0,T;E)$ by $(Af)(t)=Af(t)$, then by Theorem \ref{t1} we get the following result on the maximal $L^{p}$-regularity.

\begin{theorem}
Let $E$ be a Banach space and $A\in\mathcal{P}(\theta_{A})$ with $\theta_{A}>\frac{\pi}{2}$. Then, the following Cauchy problem
\[
f'(t)+Af(t)=g(t),\,\,\,
f(0)=0, \,\,\, \mbox{in} \,\,\, L^{p}(0,T;E), \,\,\, \mbox{with} \,\,\, g\in L^{p}(0,T;E), 
\]
$p>1$ and $T>0$ finite, has a unique solution $f\in \bigcap_{\phi<1}W^{\phi,p}(0,T;E)\cap L^{p}(0,T;(E,\mathcal{D}(A))_{\phi,q})$, for any $q\geq1$, depending continuously on $g$, which is given by
\begin{gather*}
f(t)=\frac{1}{2\pi i}\int_{\Gamma_{\theta_{A}}}\int_{0}^{t}(A+z)^{-1}e^{z(t-x)}g(x)dxdz, \,\,\,\,\,\, \forall t\in [0,T].
\end{gather*}
If  $g\in \bigcup_{\phi>0}W^{\phi,p}(0,T;E)\cup L^{p}(0,T;(E,\mathcal{D}(A))_{\phi,q})$ for some $q\geq1$, then $f\in W^{1,p}(0,T;E)\cap L^{p}(0,T;\mathcal{D}(A))$. Moreover, if $A$ admits a bounded $H^{e,\infty}_{\mathcal{L}(L^{p}(0,T;E))}(\theta_{A})$-calculus, then $f\in W^{1,p}(0,T;E)\cap L^{p}(0,T;\mathcal{D}(A))$ for any $g\in L^{p}(0,T;E)$.
\end{theorem}

By the Riemann-Lebesgue lemma (see section III.4.2 in \cite{Am}), the Fourier transform of a function in $L^{1}(\mathbb{R};E)$ vanish at infinity. Thus, if $\mathrm{Re}(\lambda)\geq k$ for some $k\in\mathbb{R}$, then
\begin{gather}\label{e234}
\lim_{\lambda\rightarrow\infty}\|(B+\lambda)^{-1}g\|=0, \,\,\, \forall g\in L^{p}(0,T;E).
\end{gather}
Also, by integration by parts, we find the relation
\begin{gather*}
\lambda(B+\lambda)^{-1}g(t)=g(t)-(B+\lambda)^{-1}Bg(t), \,\,\,  \forall g\in W^{1,p}(0,T;E), \,\,\, g(0)=0.
\end{gather*}
Hence, for  $\mathrm{Re}(\lambda)\geq k$, $k\in\mathbb{R}$, we have that
\begin{gather}\label{e233}
\lim_{\lambda\rightarrow\infty}|\lambda|\|(B+\lambda)^{-1}g\|<\infty, \,\,\, \forall g\in W^{1,p}(0,T;E), \,\,\, g(0)=0.
\end{gather}

\section{The abstract hyperbolic problem}

\begin{definition}\label{dp}
Let $E$ be a Banach space and $c>0$. Let $\mathcal{Q}(c)$ be the class of closed linear operators in $E$ such that if $A\in\mathcal{Q}(c)$, then $A\in\mathcal{P}(0)$,
\[
\Pi_{c}=\{z\in\mathbb{C}\,|\, \mathrm{Re}(z)\geq c-\frac{(\mathrm{Im}(z))^{2}}{4c}\}\subset\rho{(-A)}
\]
and
\[
\|(A+z)^{-1}\|=o(1), \,\,\, |z|^{\frac{1}{2}}\|(A+z)^{-1}A^{-\frac{1}{2}}\|=o(1) \,\,\,\,\,\, \mbox{in} \,\,\,\,\,\, \Pi_{c}.
\]
\end{definition}

In the following we denote $W_{0}^{k,p}(0,T;\mathcal{D}(A^{\alpha}))=B^{-k}A^{-\alpha}L^{p}(0,T;E)$, $k\in\mathbb{N}$, $\alpha\geq0$, where $B^{-1}$ is defined in (\ref{e778}). For the above class of operators, we have the following.

\begin{theorem}\label{t2}
Let $E$ be a Banach space and $A\in\mathcal{Q}(c^{2})$ for some $c>0$. Then, the following Cauchy problem
\begin{gather*}
f''(t)+Af(t)=g(t), \,\,\, f(0)=f'(0)=0 \,\,\, \mbox{in} \,\,\, L^{p}(0,T;E), \\
\mbox{with}\\
g\in L^{p}(0,T;\mathcal{D}(A^{\frac{3}{2}}))\cap W_{0}^{2,p}(0,T;\mathcal{D}(\sqrt{A})),
\end{gather*}
$p>1$ and $T>0$ finite, has (in the sense of (\ref{satyrdark}) and (\ref{Qdark})) a unique solution $f\in W^{1,p}(0,T;E)$ given by
\begin{gather*}
f(t)=\frac{1}{2\pi i}\int_{i\mathbb{R}-c}\int_{0}^{t}(A+z^{2})^{-1}e^{z(x-t)}g(x)dxdz, \,\,\, \forall t\in[0,T].
\end{gather*}
\end{theorem}

\begin{proof}
If $B$ is the operator from the previous section, then the problem becomes 
\begin{gather*}
(B^{2}+A)f=g.
\end{gather*}
By the relation
\begin{gather*}
(\pm i\sqrt{A}+z)^{-1}=z(A+z^{2})^{-1}\mp i\sqrt{A}(A+z^{2})^{-1},
\end{gather*}
we see that $(\pm i\sqrt{A}+z)^{-1}$ is defined on the area $\{z\in\mathbb{C}\,|\,\mathrm{Re}(z)\leq -c\}$, since the last is mapped by the $z\rightarrow z^{2}$ to $\Pi_{c^{2}}$. Also, by the equation
\begin{gather*}
(\pm i\sqrt{A}+z)^{-1}A^{-\frac{1}{2}}=z(A+z^{2})^{-1}A^{-\frac{1}{2}}\mp i(A+z^{2})^{-1},
\end{gather*}
we have that 
\begin{gather}\label{satyr}
\|(\pm i\sqrt{A}+z)^{-1}A^{-\frac{1}{2}}\|=o(1) \,\,\,\,\,\, \mbox{in} \,\,\,\,\,\, \{z\in\mathbb{C}\,|\,\mathrm{Re}(z)\leq -c\}.
\end{gather}
Let $c'>c>0$ and $g\in L^{p}(0,T;\mathcal{D}(A^{\frac{3}{2}}))\cap W_{0}^{2,p}(0,T;\mathcal{D}(\sqrt{A}))$. By (\ref{e234}), (\ref{satyr}) and Cauchy's theorem we find that
\begin{eqnarray}\nonumber
\lefteqn{-\frac{i}{2}B^{-1}\sqrt{A}g+\frac{1}{2}g}\\\nonumber
&=&-\frac{1}{2\pi i}\int_{i\mathbb{R}-c'}(B+z)^{-1}(-i\sqrt{A}+B)g\frac{1}{2z}dz-\frac{1}{2\pi i}\int_{i\mathbb{R}-c'}(i\sqrt{A}-z)^{-1}A^{-\frac{1}{2}}A^{\frac{1}{2}}(-i\sqrt{A}+B)g\frac{1}{2z}dz\\\nonumber
&=&\frac{1}{(2\pi i)^{2}}\int_{i\mathbb{R}-c'}\int_{i\mathbb{R}-c}(B+z)^{-1}(z+\lambda)^{-1}(z-\lambda)^{-1}(-i\sqrt{A}+B)gd\lambda dz\\\nonumber
&&+\frac{1}{(2\pi i)^{2}}\int_{i\mathbb{R}-c'}\int_{i\mathbb{R}-c}(i\sqrt{A}-z)^{-1}(z+\lambda)^{-1}(z-\lambda)^{-1}(-i\sqrt{A}+B)gd\lambda dz\\\nonumber
&=&\frac{1}{(2\pi i)^{2}}\int_{i\mathbb{R}-c'}\int_{i\mathbb{R}-c}(i\sqrt{A}-z)^{-1}(B+z)^{-1}(i\sqrt{A}+B)(z+\lambda)^{-1}(z-\lambda)^{-1}(-i\sqrt{A}+B)gd\lambda dz\\\nonumber
&=&\frac{1}{(2\pi i)^{2}}\int_{i\mathbb{R}-c'}\int_{i\mathbb{R}-c}(i\sqrt{A}-z)^{-1}(B+z)^{-1}(z+\lambda)^{-1}(z-\lambda)^{-1}(B^{2}+A)gd\lambda dz,\\\label{Q1}
\end{eqnarray}
\begin{eqnarray}\nonumber
\lefteqn{0=\frac{1}{2\pi i}\int_{i\mathbb{R}-c'}(i\sqrt{A}-z)^{-1}A^{-\frac{1}{2}}(B-z)^{-1}A^{\frac{1}{2}}(B^{2}+A)g\frac{1}{2z}dz}\\\nonumber
&=&-\frac{1}{(2\pi i)^{2}}\int_{i\mathbb{R}-c'}\int_{i\mathbb{R}-c}(i\sqrt{A}-z)^{-1}(B+\lambda)^{-1}(z+\lambda)^{-1}(z-\lambda)^{-1}(B^{2}+A)gd\lambda dz,\\\label{Q2}
\end{eqnarray}
\begin{eqnarray}\nonumber
\lefteqn{\frac{i}{2}B^{-1}\sqrt{A}g+\frac{1}{2}g}\\\nonumber
&=&-\frac{1}{2\pi i}\int_{i\mathbb{R}-c'}(B+z)^{-1}(i\sqrt{A}+B)g\frac{1}{2z} dz+\frac{1}{2\pi i}\int_{i\mathbb{R}-c'}(i\sqrt{A}+z)^{-1}A^{-\frac{1}{2}}A^{\frac{1}{2}}(i\sqrt{A}+B)g\frac{1}{2z} dz\\\nonumber
&=&-\frac{1}{2\pi i}\int_{i\mathbb{R}-c'}(-i\sqrt{A}-z)^{-1}(B+z)^{-1}(-i\sqrt{A}+B)(i\sqrt{A}+B)g\frac{1}{2z} dz\\\nonumber
&=&-\frac{1}{(2\pi i)^{2}}\int_{i\mathbb{R}-c'}\int_{i\mathbb{R}-c}(i\sqrt{A}+\lambda)^{-1}A^{-\frac{1}{2}}(B+z)^{-1}(z+\lambda)^{-1}(z-\lambda)^{-1}A^{\frac{1}{2}}(B^{2}+A)gd\lambda dz\\\label{Q3}
\end{eqnarray}
and
\begin{eqnarray}\nonumber
\lefteqn{0=\frac{1}{2\pi i}\int_{i\mathbb{R}-c'}(B-z)^{-1}(i\sqrt{A}+B)g\frac{1}{2z}dz+\frac{1}{2\pi i}\int_{i\mathbb{R}-c'}(-i\sqrt{A}-z)^{-1}A^{-\frac{1}{2}}A^{\frac{1}{2}}(i\sqrt{A}+B)g\frac{1}{2z}dz}\\\nonumber
&=&-\frac{1}{(2\pi i)^{2}}\int_{i\mathbb{R}-c'}\int_{i\mathbb{R}-c}(B+\lambda)^{-1}(z+\lambda)^{-1}(z-\lambda)^{-1}(i\sqrt{A}+B)gd\lambda dz\\\nonumber
&&-\frac{1}{(2\pi i)^{2}}\int_{i\mathbb{R}-c'}\int_{i\mathbb{R}-c}(-i\sqrt{A}-\lambda)^{-1}A^{-\frac{1}{2}}(z+\lambda)^{-1}(z-\lambda)^{-1}A^{\frac{1}{2}}(i\sqrt{A}+B)gd\lambda dz\\\nonumber
&=&-\frac{1}{(2\pi i)^{2}}\int_{i\mathbb{R}-c'}\int_{i\mathbb{R}-c}(-i\sqrt{A}-\lambda)^{-1}(B+\lambda)^{-1}(-i\sqrt{A}+B)(z+\lambda)^{-1}(z-\lambda)^{-1}(i\sqrt{A}+B)gd\lambda dz\\\nonumber
&=&\frac{1}{(2\pi i)^{2}}\int_{i\mathbb{R}-c'}\int_{i\mathbb{R}-c}(i\sqrt{A}+\lambda)^{-1}(B+\lambda)^{-1}(z+\lambda)^{-1}(z-\lambda)^{-1}(B^{2}+A)gd\lambda dz.\\\label{Q4}
\end{eqnarray}
\par
Let the family $y_{\lambda,z}(h)\in L^{p}(0,T;E)$, defined for $\lambda\in i\mathbb{R}-c$, $z\in i\mathbb{R}-c'$ and for any $h\in L^{p}(0,T;E)$ by
\begin{gather*}
y_{\lambda,z}(h)=(i\sqrt{A}-z)^{-1}(-i\sqrt{A}-\lambda)^{-1}(B+z)^{-1}(B+\lambda)^{-1}h.
\end{gather*}
There is
\begin{eqnarray*}
\lefteqn{y_{\lambda,z}(h)=[(i\sqrt{A}-z)^{-1}+(-i\sqrt{A}-\lambda)^{-1}](z+\lambda)^{-1}[(B+z)^{-1}-(B+\lambda)^{-1}](z-\lambda)^{-1}h}\\
&&=(i\sqrt{A}-z)^{-1}(B+z)^{-1}(z+\lambda)^{-1}(z-\lambda)^{-1}h-(i\sqrt{A}-z)^{-1}(B+\lambda)^{-1}(z+\lambda)^{-1}(z-\lambda)^{-1}h\\
&&-(i\sqrt{A}+\lambda)^{-1}(B+z)^{-1}(z+\lambda)^{-1}(z-\lambda)^{-1}h+(i\sqrt{A}+\lambda)^{-1}(B+\lambda)^{-1}(z+\lambda)^{-1}(z-\lambda)^{-1}h.
\end{eqnarray*} 
By (\ref{Q1}), (\ref{Q2}), (\ref{Q3}) and (\ref{Q4}), we have that
\begin{gather}\label{satyrdark}
\frac{1}{(2\pi i)^{2}}\int_{i\mathbb{R}-c'}\int_{i\mathbb{R}-c}y_{\lambda,z}((B^{2}+A)g)d\lambda dz=g.
\end{gather}
For any $w\in W_{0}^{1,p}(0,T;\mathcal{D}(\sqrt{A}))$, by (\ref{e233}), (\ref{satyr}) and Cauchy's theorem we have that
\begin{eqnarray}\nonumber
\lefteqn{\frac{1}{2\pi i}\int_{i\mathbb{R}-c}(A+z^{2})^{-1}(B+z)^{-1}wdz}\\\nonumber
&=&\frac{1}{2\pi i}\int_{i\mathbb{R}-c'}\big((i\sqrt{A}+z)^{-1}-(i\sqrt{A}-z)^{-1}\big)(B+z)^{-1}\frac{1}{2z}wdz\\\nonumber
&=&-\frac{1}{2\pi i}\int_{i\mathbb{R}-c'}(i\sqrt{A}-z)^{-1}(B+z)^{-1}\frac{1}{2z}wdz+\frac{1}{2\pi i}\int_{i\mathbb{R}-c'}(i\sqrt{A}-z)^{-1}(B-z)^{-1}\frac{1}{2z}wdz\\\nonumber
&&+\frac{1}{2\pi i}\int_{i\mathbb{R}-c'}(i\sqrt{A}+z)^{-1}(B+z)^{-1}\frac{1}{2z}wdz\\\nonumber
&=&\frac{1}{(2\pi i)^{2}}\int_{i\mathbb{R}-c'}\int_{i\mathbb{R}-c}(i\sqrt{A}-z)^{-1}(B+z)^{-1}(z+\lambda)^{-1}(z-\lambda)^{-1}wd\lambda dz\\\nonumber
&&-\frac{1}{(2\pi i)^{2}}\int_{i\mathbb{R}-c'}\int_{i\mathbb{R}-c}(i\sqrt{A}-z)^{-1}(B+\lambda)^{-1}(z+\lambda)^{-1}(z-\lambda)^{-1}wd\lambda dz\\\nonumber
&&-\frac{1}{(2\pi i)^{2}}\int_{i\mathbb{R}-c'}\int_{i\mathbb{R}-c}(i\sqrt{A}+\lambda)^{-1}(B+z)^{-1}(z+\lambda)^{-1}(z-\lambda)^{-1}wd\lambda dz\\\nonumber
&&+\frac{1}{(2\pi i)^{2}}\int_{i\mathbb{R}-c'}\int_{i\mathbb{R}-c}(i\sqrt{A}+\lambda)^{-1}(B+\lambda)^{-1}(z+\lambda)^{-1}(z-\lambda)^{-1}wd\lambda dz\\\label{Qdark}
&=&\frac{1}{(2\pi i)^{2}}\int_{i\mathbb{R}-c'}\int_{i\mathbb{R}-c}y_{\lambda,z}(w)d\lambda dz.
\end{eqnarray}
Hence, by (\ref{satyrdark}) we find the following solution to the problem
\begin{gather*}
\frac{1}{2\pi i}\int_{i\mathbb{R}-c}(A+z^{2})^{-1}(B+z)^{-1}gdz,
\end{gather*}
and the final expression follows by (\ref{e778}). 
\end{proof}
 
\begin{example}
Let $\mathcal{H}$ be a Hilbert space and $A$ a strictly positive self-adjoint operator on $\mathcal{H}$. Then, by the spectral theorem, we can easily see that $A\in\mathcal{Q}(c)$, where $c$ is the lower bound of the spectrum. Hence, we can apply the previous theorem to get a solution to the second order problem for suitable $g$. 
\end{example}

\section{Appendix}

The following well known argument can also be found in section III.4.6 in \cite{Am}.\\

\textbf{\textit{Sectoriality extension argument:} }
Let $\Omega\subset\mathbb{C}$ be closed and connected, and $A$ be a closed linear operator in a Banach space $E$ such that $\Omega\subset\rho(-A)$ and $(1+|\lambda|)\|(A+\lambda)^{-1}\|\leq K$, for any $\lambda\in\Omega$ and some $K\geq1$. Let the set 
\begin{gather*}
\Omega'=\cup_{\lambda\in\Omega}\{z\in\mathbb{C}\,|\,|z-\lambda|\leq(1+|\lambda|)/2K\}.
\end{gather*}
Then, by the relation
\begin{gather*}
A+z=(A+\lambda)\big(I+(z-\lambda)(A+\lambda)^{-1}\big),
\end{gather*}
we have that $(A+z)^{-1}$ is defined in $\Omega'$ and  
\begin{eqnarray*}
\lefteqn{\|(A+z)^{-1}\|\leq \|\big(I+(z-\lambda)(A+\lambda)^{-1}\big)^{-1}\|\|(A+\lambda)^{-1}\|\leq\frac{2K}{1+|\lambda|}}\\
&&\leq \frac{2K(1+|\lambda|+|z-\lambda|)}{(1+|\lambda|)(1+|z|)}\leq \frac{2K}{(1+|z|)}(1+\frac{1}{2K})=\frac{2K+1}{1+|z|}.
\end{eqnarray*}


\begin{thebibliography}{99}

\bibitem{Am} H. Amann, {\em Linear and quasilinear parabolic problems}. Monographs in Mathematics Vol. 89, Birkh\"auser Verlag (1995).

\bibitem{CL} P. Clément and S. Li, {\em Abstract parabolic quasilinear equations and application to a groundwater flow problem}. Adv. Math. Sci. Appl. 3, Special Issue, 17--32 (1993/94).

\bibitem{DV} G. Dore and A. Venni, {\em On the closedness of the sum of two closed operators}. Math. Z. 196, no. 2, 189--201(1987).

\bibitem{KW} N. Kalton and L. Weis, {\em The $H^{\infty}$-calculus and sums of closed operators}. Math. Ann. 321, no. 2, 319--345 (2001).

\bibitem{LA} F. Lancien, G. Lancien and C. Le Merdy, {\em A joint functional calculus for sectorial operators with commuting resolvents}. Proc. London Math. Soc. (3) 77, no. 2, 387--414 (1998).

\bibitem{NZ} H. Neidhardt and V. Zagrebnov, {\em Linear non-autonomous Cauchy problems and evolution semigroups}. Adv. Differential Equations 14, no. 3-4, 289--340 (2009).  

\bibitem{PG} G. Da Prato and P. Grisvard, {\em Sommes d'opérateurs linéaires et équations différentielles opérationnelles}. J. Math. Pures Appl. (9) 54, no. 3, 305--387 (1975).

\end{thebibliography}
\end{document}